\documentclass[reqno]{amsart}
\usepackage{a4wide}
\usepackage[american]{babel}
\usepackage{amssymb, amsmath, amsthm,amsfonts}
\usepackage[utf8]{inputenc}
\usepackage{tikz}
\usetikzlibrary{matrix,arrows,positioning,backgrounds,shapes}
\usepackage{multirow}
\usepackage{aliascnt}
\usepackage{nicefrac}
\usepackage{booktabs}

\newcommand{\V}{\ensuremath{\mathrm{V}}}
\newcommand{\ZZ}{\ensuremath{\mathbb{Z}}}
\newcommand{\QQ}{\ensuremath{\mathbb{Q}}}

\newcommand{\Ker}{\ensuremath{\operatorname{Ker}}}
\newcommand{\sgn}{\ensuremath{\operatorname{sgn}}}
\newcommand{\GL}{\ensuremath{\operatorname{GL}}}
\newcommand{\PSL}{\ensuremath{\operatorname{PSL}}}

\newcommand{\Inn}{\ensuremath{\operatorname{Inn}}}
\newcommand{\Aut}{\ensuremath{\operatorname{Aut}}}

\newlength{\heightofhw}
\newcommand{\eigbox}[2]{\ensuremath{#1 \times \boxed{\rule{0cm}{\heightofhw} #2}}}
\newcommand{\cc}[1]{\ensuremath{{\texttt{#1}}}}

\definecolor{LinkColor}{rgb}{0,0,0} %black
\usepackage[colorlinks=true,linkcolor=LinkColor,citecolor=LinkColor,urlcolor= LinkColor, naturalnames, hyperindex, pdfstartview=FitH, bookmarksnumbered, plainpages]{hyperref}

\usepackage[capitalise,nameinlink]{cleveref}

\newtheorem{theorem}{Theorem}[section]
\crefname{theorem}{Theorem}{Theorems}

\newaliascnt{corollary}{theorem}
\newtheorem{corollary}[corollary]{Corollary}
\aliascntresetthe{corollary}
\crefname{corollary}{Corollary}{Corollaries}

\newaliascnt{lemma}{theorem}
\newtheorem{lemma}[lemma]{Lemma}
\aliascntresetthe{lemma}
\crefname{lemma}{Lemma}{Lemmas}

\newaliascnt{proposition}{theorem}
\newtheorem{proposition}[proposition]{Proposition}
\aliascntresetthe{proposition}
\crefname{proposition}{Proposition}{Propositions}

\newtheorem*{zc*}{Zassenhaus Conjecture (ZC)}
\newtheorem*{pq*}{Prime Graph Question (PQ)}
\newtheorem*{theorem*}{Theorem}

\theoremstyle{definition}

\newaliascnt{definition}{theorem}
\newtheorem{definition}[definition]{Definition}
\aliascntresetthe{definition}
\crefname{definition}{Definition}{Definitions}

\theoremstyle{remark}

\newaliascnt{nexample}{theorem}
\newtheorem{nexample}[nexample]{Example}
\aliascntresetthe{nexample}
\crefname{nexample}{Example}{Examples}

\newaliascnt{remark}{theorem}
\newtheorem{remark}[remark]{Remark}
\aliascntresetthe{remark}
\crefname{remark}{Remark}{Remarks}

\newaliascnt{notation}{theorem}

\aliascntresetthe{notation}
\crefname{notation}{Notation}{Notations}

\title{On the Prime Graph Question for Almost Simple Groups with an Alternating Socle}
\author{Andreas B\"achle}
\address{Department of Mathematics, Vrije Universiteit Brussel,
Pleinlaan 2, 1050 Brussels, Belgium}
\email{\href{mailto:abachle@vub.ac.be}{abachle@vub.ac.be}}
\author{Mauricio Caicedo}
\email{\href{mailto:mcaicedo@vub.ac.be}{mcaicedo@vub.ac.be}}


\thanks{The research is supported by the Research Foundation Flanders (FWO - Vlaanderen), partially by the FWO project G.0157.12N}

\begin{document}

\begin{abstract}
 Let $G$ be an almost simple group with socle $A_n$, the alternating group of degree $n$. We prove that there is a unit of order $pq$ in the integral group ring of $G$ if and only if there is an element of that order in $G$ provided $p$ and $q$ are primes greater than $\frac{n}{3}$. We combine this with some explicit computations to verify the Prime Graph Question for all almost simple groups with socle $A_n$ if $n \leq 17$.
\end{abstract}

\maketitle

\section{Introduction}
 Let $\V(\ZZ G)$ denote the group of normalized units (i.e.\ units of augmentation one) of the integral group ring $\ZZ G$ of a finite group $G$. One of the most outstanding questions in integral group rings is the so-called (first) Zassenhaus Conjecture.
 
 \begin{zc*}[\cite{Zassenhaus}] Let $G$ be a finite group and $u$ a torsion unit in $\V(\ZZ G)$. Then there exists a unit $x$ in the rational group algebra $\mathbb{Q}G$ such that $x^{-1}ux = g$ for some $g \in G$.
 \end{zc*}
 
 The conjecture is known to hold for several classes of groups, e.g.\ in Weiss' celebrated article \cite{WeissCrelle} the Zassenhaus Conjecture is proved for all nilpotent groups. The second author, Á.~del~Río and L.~Margolis confirmed it for all cyclic-by-abelian groups \cite{CyclicByAbelian}. Also, for some specific non-solvable groups some results are known. For example, (ZC) holds for all $\PSL(2,q)$ with $q \leq 25$, $q \in \{31, 32\}$ \cite[Theorem~C]{BM4pII} or $q$ a Fermat or Mersenne prime \cite[Theorem~1.1]{MdRS}; including $A_5 \simeq \PSL(2,5)$ \cite{LP89} and $A_6 \simeq \PSL(2,9)$ \cite{HertweckA6}. However, in general the conjecture remains open. Because of the difficulty of the problem, W.~Kimmerle proposed to study the Prime Graph Question as a first approximation.
 
 \begin{pq*}[\cite{KimmiPQ}] Let $G$ be a finite group. Let $p$ and $q$ be different primes such that $\mathrm{V}(\mathbb{Z}G)$ contains an element of order $pq$. Does then $G$ possess an element of order $pq$?\end{pq*}
 
 In other words, this question asks whether $G$ and $\mathrm{V}(\mathbb{Z}G)$ have the same prime graph. Here, the prime graph of a group $X$ is the graph having as vertices those primes $p$ such that there is an element of order $p$ in $X$ and two different vertices $p$ and $q$ are connected if and only if there is an element of order $pq$ in $X$.

 There is an affirmative answer to the Prime Graph Question for all solvable groups and all Frobenius groups \cite{KimmiPQ}. To attack (PQ) in general, W. Kimmerle and A. Konovalov \cite[Theorem~2.1]{KiKo16} proved a reduction result to almost simple groups (groups $A$, that are sandwiched between a simple group $S$ as socle and its automorphism group, $S \simeq \Inn (S) \leq A \leq \Aut(S)$). Their result puts a spotlight on the Prime Graph Question for almost simple groups. In a series of papers, Bovdi, Konovalov et.al.\ proved (PQ) for half of the sporadic simple groups; the almost simple groups containing them were checked by Kimmerle and Konovalov (see \cite[Section 5]{KiKoStAndrews} for an overview and references). The first infinite series for which (PQ) was answered are the groups $\PSL(2,p)$, with $p$ a prime, see \cite{HertweckBrauer}. In the meantime  there is also an affirmative answer to (PQ) for almost simple groups with socle $\PSL(2,p^f)$, $f \leq 2$ \cite[Theorem~A]{BM4pI}. %Also for all almost simple groups with at most $3$ primes dividing their order, the Prime Graph Question was answered positively \cite{BaMaM10, KiKo16}.
 The Prime Graph Question also has a positive answer for all groups with an order divisible by at most 3 different primes \cite{BaMaM10, KiKo16}.
 
 In this article we consider the Prime Graph Question for almost simple groups having $A_n$ as socle. The Zassenhaus Conjecture is only known to hold for the alternating groups of degree at most $6$ \cite{HertweckA6} and the symmetric groups of degree at most $5$ \cite{ZasS5}. Up to now, the Prime Graph Question was known to have an affirmative answer for alternating groups up to degree $10$ and all almost simple groups having an alternating socle of degree at most $6$ (cf. \cite{SalimA10, KiKoStAndrews} and the references therein).  
  
 The main result of this article (contained in \cref{big_p_q,cor_A_n}) is the first generic result on the non-existence of certain units in this context.
 
 \begin{theorem*} Let $n \in \ZZ_{\geq 1}$, $p, q$ two primes with $p, q > \frac{n}{3}$ and $G$ be an almost simple group with socle $A_n$. Then $G$ possesses an element of order $pq$ if and only if there is an element of order $pq$ in $\V(\ZZ G)$.  
 \end{theorem*}
 
 We use this together with some computations to give a positive answer to the Prime Graph Question for all almost simple groups having $A_n$ as socle if $7 \leq n \leq 17$, so that (PQ) is known for this kind of groups if $n \leq 17$ (\cref{prop:PQ_for_small_asg}).

 \section{Preliminaries}
 
 We will recall the facts about symmetric groups needed here, for a more detailed account we refer to the standard references \cite{FultonHarris,James}.  %JamesKerber
 
 Let $S_n$ denote the symmetric group of degree $n$. The conjugacy classes of $S_n$ can be labeled by the partitions of $n$. For a prime $r$ and a positive integer $j$ with $rj \leq n$ we denote by $\cc{r.j}$ the conjugacy class of $S_n$ corresponding to the partition $(r^j, 1^{n -rj})$, the permutations consisting of $j$ disjoint $r$-cycles. Note that there are exactly $\lfloor \frac{n}{r} \rfloor$ conjugacy classes of elements of prime order $r$ in $S_n$. For two different primes $p$ and $q$ there is an element of order $pq$ in $S_n$ if and only if $p + q \leq n$.
 
 Also the ordinary irreducible representations are in natural bijection with the partitions of $n$. For a partition $\lambda$ of $n$, we will denote by $\chi_\lambda$ the ordinary irreducible character afforded by the Specht module $S^\lambda$. In particular, $\chi_{(n)} = 1$ is the principal character of $S_n$ and $\chi_{(1^n)} = \sgn$ is the alternating character of $S_n$.  Set $\pi = \chi_{(n-1,1)}$. Then $\pi$ is the so-called ``deleted permutation character'' (i.e.\ the character of $S_n$ coming from its natural action on $n$ points minus its trivial constituent).  Clearly, $\pi(1) = n-1$ and for a prime $r$ and $j\in \{1, ..., \lfloor \frac{n}{r} \rfloor\}$ we have $\pi(\cc{r.j}) = \pi(1) -rj$.
 
 Let $\chi$ be an ordinary character of the symmetric group and let $D$ be an integral representation affording $\chi$. We can extend $D$ linearly to a representation of the integral group ring $\ZZ S_n$ and then restrict it to a representation $\V(\ZZ S_n) \to \GL_{d}(\ZZ)$ for $d = \chi(1)$.  We will denote this representation also by $D$ and its character by $\chi$. 
 
 We will use the so-called Luthar-Passi method which allows to calculate the multiplicities of eigenvalues of torsion units under a representation.
 
 \begin{proposition}[Luthar, Passi {\cite[Theorem 1]{LP89}}]\label{LuPa}  Let $G$ be a finite group and $u \in \V(\ZZ G)$ a torsion unit of order $k$. Let $\zeta \in \mathbb{C}^\times$ be a primitive $k$-th root of unity. 
 Let $\chi$ be an ordinary character of $G$ and let $D$ be a representation affording $\chi$. Then the multiplicity of $\zeta^\ell$ as an eigenvalue of $D(u)$ is given by
  \begin{equation*} \mu_\ell(u, \chi) = \frac{1}{k}\sum_{d \mid k} \operatorname{Tr}_{\QQ(\zeta^d)/\QQ}(\chi(u^d)\zeta^{-d \ell}).\end{equation*} 
 \end{proposition}
  
 Quite some of information about a torsion unit is encoded in its partial augmentations: For an element $u = \sum_{g \in G} u_g g \in \ZZ G$ and a conjugacy class $C$ of $G$, $\varepsilon_C(u) = \sum_{g \in C} u_g$ denotes its \emph{partial augmentation} at the conjugacy class $C$.
 The last proposition can be used to show rational conjugacy of torsion units in the following way: two torsion units $u, v \in\ZZ G$ of the same order $k$, are rationally conjugate, if and only if $u^d$ and $v^d$ have the same partial augmentations for all divisors $d$ of $k$ \cite[Theorem 2.5]{MRSW}. In particular there is no torsion unit of order $k$ in $\V(\ZZ G)$ with a prescribed tuple of partial augmentations $(\varepsilon_C(u^d))_{C, d}$, where $C$ runs through the conjugacy classes of $G$ and $d$ runs through the divisors of $k$, if $\mu_\ell(u, \chi)$  is not integral or negative for some $\ell$ and some character $\chi$ of $G$. Certain partial augmentations are a priori known to be zero: for a torsion unit $u \in \V(\ZZ G)$ of order $k$, $\varepsilon_1(u) = 0$ if $k \not= 1$ \cite[Proposition (1.4)]{SehgalBook2} and $\varepsilon_C(u) = 0$, whenever $k$ does not divide the order of $x$ for some $x \in C$ \cite[Theorem 2.3]{HertweckBrauer}. 
  
 We will repeatedly need the multiplicities of certain roots of unity under the representation $P$ corresponding to the character $\pi = \chi_{(n-1,1)}$, so we will collect them in the following lemma.
 
 \begin{lemma}\label{prop:mu1pi} \begin{enumerate} 
 \item \label{prop:mu1rpi} Let $r$ be a prime and let $v \in \V(\ZZ S_n)$ be a unit of order $r$. The multiplicity of a primitive $r$-th root of unity as eigenvalue of $P(v)$ equals $\mu_1(v, \pi) = \sum_{j = 1}^{\lfloor \frac{n}{r} \rfloor} j \varepsilon_{\cc{r.j}}(v)$.
 \item \label{prop:mu1pqpi}  Let $u \in \V(\ZZ S_n)$ be an element of order $pq$ for two different primes $p, q$ with $p + q > n$. Then the multiplicity of a primitive $pq$-th root of unity as an eigenvalue of $P(u)$ is given by \begin{equation*} \mu_1(u, \pi) =  \frac{1}{pq} \Bigg[ q  \sum_{j = 1}^{ \lfloor \frac{n}{q} \rfloor } j \left( \varepsilon_{\cc{q.j}}(u^p)  -  \ \varepsilon_{\cc{q.j}}(u) \right)  +  p \sum_{k = 1}^{ \lfloor \frac{n}{p} \rfloor } k \left( \varepsilon_{\cc{p.k}}(u^q) -  \varepsilon_{\cc{p.k}}(u) \right)  \Bigg]. \end{equation*}
 \end{enumerate}
 \end{lemma}

 \begin{proof} \begin{enumerate} \item By what we noted above, the element $v$ can only have non-zero partial augmentations at conjugacy classes of elements of order $r$. The multiplicity of a primitive $r$-th root of unity of $v$ under the representation $P$ can be calculated by the Luthar-Passi-formula (\cref{LuPa}) to be  \[ \mu_1(v, \pi) = \frac{1}{r} \Big( \pi(1) - \sum_{j = 1}^{\lfloor \frac{n}{r} \rfloor} \varepsilon_{\cc{r.j}}(v) (\pi(1) - rj) \Big) = \sum_{j = 1}^{\lfloor \frac{n}{r} \rfloor} j \varepsilon_{\cc{r.j}}(v), \] using that $v$ is normalized. 
 \item By assumption, $S_n$ has no elements of order $pq$. Using the Luthar-Passi-formula and the character values of $\pi$ we get 
 \begin{align*}\begin{split} \mu_1(u, \pi)  & = \frac{1}{pq} \Bigg[ \pi(1) - \pi(u^q) - \pi(u^p) + \pi(u)   \Bigg] \\ & =  \frac{1}{pq} \Bigg[ \pi(1) - \Big(\pi(1) - q \sum_{j = 1}^{ \lfloor \frac{n}{q} \rfloor } j \varepsilon_{\cc{q.j}}(u^p)\Big) - \Big(\pi(1) - p\sum_{k = 1}^{ \lfloor \frac{n}{p} \rfloor } k \varepsilon_{\cc{p.k}}(u^q)\Big)  \\ & \hspace{3cm} +  \Big(\pi(1) -  q\sum_{j = 1}^{ \lfloor \frac{n}{q} \rfloor } j \varepsilon_{\cc{q.j}}(u) -  p \sum_{k = 1}^{ \lfloor \frac{n}{p} \rfloor } k \varepsilon_{\cc{p.k}}(u) \Big)  \Bigg]  \\ & =  \frac{1}{pq} \Bigg[ q  \Big( \sum_{j = 1}^{ \lfloor \frac{n}{q} \rfloor } j \varepsilon_{\cc{q.j}}(u^p)  -  \sum_{j = 1}^{ \lfloor \frac{n}{q} \rfloor } j \varepsilon_{\cc{q.j}}(u) \Big)  +  p \Big(\sum_{k = 1}^{ \lfloor \frac{n}{p} \rfloor } k \varepsilon_{\cc{p.k}}(u^q) - \sum_{k = 1}^{ \lfloor \frac{n}{p} \rfloor } k \varepsilon_{\cc{p.k}}(u) \Big)  \Bigg].\qedhere \label{eq:mu1chi} \end{split}  \end{align*} % where we exploited for the second equality the fact that that $u$ is normalized. 
 \end{enumerate} \end{proof}

 \begin{lemma}\label{prop:mu1is0} Let $n \geq 7$ be an integer and let $p, q$ be primes with $p > \frac{n}{2}$ and $n \geq q \geq 3$. Assume $u \in \V(\ZZ S_n)$ is of order $pq$. Then $\mu_1(u, \pi) = 0$ and $\mu_q(u, \pi) = 1$. If, additionally, $p + q \geq n$, we have $\mu_p(u, \pi) = \mu_1(u^p, \pi) \in \{0, 1\}$ and $1$ is possible only if $p + q \in \{n, n+1\}$. \end{lemma}
 
 \begin{proof} Let $P$ be a integral representation affording $\pi = \chi_{(n-1,1)}$ and extend it as explained above to a representation of $\V(\ZZ S_n)$. We can diagonalize $P(u)$ over $\mathbb{C}$. As $P(u)$ is a rational matrix, we can write
\[ P(u) \sim \text{diag}\left( \eigbox{m_1}{1},\ \eigbox{m_q}{\zeta_q, ..., \zeta_q^{q-1}},\ \eigbox{m_p}{\zeta_p, ....., \zeta_p^{p-1}},\ \eigbox{m_{pq}}{\zeta_{pq}, ....., \zeta_{pq}^{(p-1)(q-1)}} \right) \] i.e.\ \begin{itemize}
\item $m_1 = \mu_0(u, \pi)$ is the number of $1$'s as eigenvalues of $P(u)$,
\item $m_q = \mu_p(u, \pi)$ the number of $(q-1) \times (q-1)$-blocks with primitive $q$-th roots of unity,
\item $m_p = \mu_q(u, \pi)$ the number of $(p-1) \times (p-1)$-blocks with primitive $p$-th roots of unity,
\item $m_{pq} = \mu_1(u, \pi)$ the number of $(q-1)(p-1) \times (q-1)(p-1)$-blocks with primitive $pq$-th roots of unity.
\end{itemize} Then we have \[ P(u^q) \sim \text{diag}\left( \eigbox{(m_1 + (q-1)m_q)}{1},\ \eigbox{(m_p + (q-1)m_{pq})}{\zeta_p, ....., \zeta_p^{p-1}} \right). \] On the other hand, as there is only one conjugacy class, $\cc{p.1}$, with elements of order $p$ in $S_n$, we have that $u^q$ is rationally conjugate to an element in $\cc{p.1}$ and hence there is exactly one block with primitive $p$-th roots of unity in the diagonalized version of $P(u^q)$, forcing $m_p + (q-1)m_{pq} = 1$, thus $\mu_1(u, \pi) = m_{pq} = 0$ and $\mu_q(u, \pi) = m_{p} = 1$, as $q \geq 3$.

Assume now that $p + q \geq n$. Then $n - 1 = \pi(1) = m_1 + m_q(q-1) + (p-1)$, forcing $m_q \leq 1$. As $m_{pq} = 0$ we get $\mu_1(u^p, \pi) = m_q$.\end{proof}

\section{Units of Odd Order}

In this section we will deal with units of order $pq$ for odd primes $p$ and $q$. We will prove the main result of this article, but we will 
first give an easy example.

\begin{nexample}\label{ex:order_3_5} We show that there is no unit of order $3 \cdot 5$ in the unit group of $\ZZ S_7$.

Assume by means of contradiction that $u$ is such a unit. By the remarks in the preliminaries, $u$ can only have non-zero partial augmentations at the classes $\cc{3.1}$, $\cc{3.2}$, and $\cc{5.1}$. Consider the character $\chi$ belonging to the hook partition $(4, 1^3)$ of $7$. By Frobenius' formula \cite[4.10]{FultonHarris} it has the following values:
\begin{center}
 \noindent\begin{tabular}{ccccc} \toprule 
        & $1$ & $\cc{3.1}$ & $\cc{3.2}$ & $\cc{5.1}$ \\ \midrule \\[-.38cm]
   $\chi = \chi_{(4, 1^3)}$ & $20$ & $2$ & $2$ & $0$ \\
   \bottomrule
 \end{tabular}
\end{center}
As the character is constant on all classes of elements order $3$, the multiplicities of the eigenvalues do not depend on the partial augmentations of $u^5$ (and $u^3$). Using that $u$ has augmentation $1$, we get by the Luthar-Passi-formula (\cref{LuPa}) for the multiplicity of $1$ and of a primitive third root of unity as an eigenvalue of $u$ under a representation affording $\chi$ \[
  \begin{tabular}{rcrcl}
  $\mu_0(u, \chi)$ & $=$ & $-\frac{16}{15}\varepsilon_{\cc{5.1}}(u)$ & $+\frac{8}{3}$ \\[.1cm]
  $\mu_5(u, \chi)$ & $=$ & $\frac{8}{15}\varepsilon_{\cc{5.1}}(u)$ & $+\frac{2}{3}$.
  \end{tabular}
  \]
It is easy to see that there is no $\varepsilon_{\cc{5.1}}(u) \in \ZZ$ such that both multiplicities are non-negative integers. 
\end{nexample}

The idea to use characters constant on conjugacy classes of the involved orders was introduced in \cite{BKHS}.  However it turns out that in the setting of symmetric groups the degrees of such characters are in general too big to provide enough information to exclude the existence of the units in question. Hence the strategy to prove the following theorem has to be different.

\begin{theorem}\label{big_p_q} Let $n \in \ZZ_{\geq 1}$ and $p, q$ be two primes with $p, q > \frac{n}{3}$. Then there is an element of order $pq$ in $S_n$ if and only if there is an element of order $pq$ in $\V(\ZZ S_n)$. \end{theorem}

 \begin{proof} It clearly is enough to prove the sufficiency. Assume that $u \in \V(\ZZ S_n)$ is a unit of order $pq$ and there is no element of order $pq$ in $S_n$, i.e.\ $p + q > n$. By \cite[Corollary 4.1]{CL} there is an element of order $p^2$ in $S_n$ if and only if there is such an element in $\V(\ZZ S_n)$. Hence, we may from now on assume without loss of generality that $p > q$. By \cite{KiKoStAndrews} the Prime Graph Question has a positive answer for $S_n$ if $n \leq 6$. Together with \cref{ex:order_3_5} this implies that we can assume $p \geq 7$ and $q \geq 3$. We can also assume that $p > \frac{n}{2}$, for otherwise there is an element of order $pq$ in $S_n$. Thus there is only one conjugacy class, $\cc{p.1}$, of elements of order $p$ in $S_n$ and $u^q$ is rationally conjugate to an element in that class. We will distinguish two cases: $q > \frac{n}{2}$ or $\frac{n}{2} \geq q > \frac{n}{3}$.

Assume $\lfloor \frac{n}{q} \rfloor = 1$: In this case also the element $u^p$ of order $q$ is rationally conjugate to a group element which lies in the class $\cc{q.1}$. Now $\mu_1(u, \pi)$, as calculated in \cref{prop:mu1pi}\eqref{prop:mu1pqpi}, simplifies to \[\mu_1(u, \pi) =  \frac{1}{pq} \Bigg[ q  \Big( 1  - \varepsilon_{\cc{q.1}}(u) \Big)  +  p \Big( 1 -  \varepsilon_{\cc{p.1}}(u) \Big) \Bigg] = \frac{1}{pq} \Bigg[ \varepsilon_{\cc{p.1}}(u)(q-p) + p  \Bigg], \] using that $1 = \varepsilon(u) = \varepsilon_{\cc{p.1}}(u) + \varepsilon_{\cc{q.1}}(u)$. As $\mu_1(u, \pi)  = 0$ by \cref{prop:mu1is0}, we get $\varepsilon_{\cc{p.1}}(u) = \frac{p}{p-q}$, which has to be an integer, hence $(p, q) = (3, 2)$ and $n = 3$, a case for which the result is known.

Assume $\lfloor \frac{n}{q} \rfloor = 2$: By \cref{prop:mu1pi}\eqref{prop:mu1rpi} we get $\mu_1(u^p, \pi) = \varepsilon_{\cc{q.1}}(u^p) + 2\varepsilon_{\cc{q.2}}(u^p)$ and by \cref{prop:mu1is0} this has to be equal to $0$ or $1$. Note in particular that $\mu_1(u^p, \pi) = 1$ is possible only if $n = p + q -1$. Together with $\varepsilon_{\cc{q.1}}(u^p) + \varepsilon_{\cc{q.2}}(u^p) = \varepsilon(u^p) = 1$ we get the following possible partial augmentations for $u^p$:
\begin{align*} & (\varepsilon_{\cc{q.1}}(u^p), \varepsilon_{\cc{q.2}}(u^p)) = (2, -1) \\ \text{or} \quad & (\varepsilon_{\cc{q.1}}(u^p), \varepsilon_{\cc{q.2}}(u^p)) = (1, 0)\quad  \text{ and } n = p + q -1.\end{align*}
 
 We now have to use another character, $\rho = \chi_{(n - 2, 1, 1)}$. Its values can be computed e.g.\ using Frobenius' formula for character values \cite[4.10]{FultonHarris}. For the convenience of the reader we state the relevant values in the following table
\begin{center}
  {\footnotesize \noindent\begin{tabular}{ccccc} \toprule
        & $1$ & $\cc{q.1}$ & $\cc{q.2}$ & $\cc{p.1}$ \\ \midrule \\[-.38cm]
   $\pi = \chi_{(n-1,1)}$ & $n-1$ & $\pi(1)-q$ & $\pi(1)-2q$ & $\pi(1)-p$ \\
   $\rho = \chi_{(n-2,1,1)}$ & $\frac{1}{2}(n-1)(n-2)$ & $\rho(1) - \frac{1}{2}q(2n-q-3)$ & $\rho(1) - q(2n-2q-3)$ & $\rho(1) - \frac{1}{2}p(2n-p-3)$ \\
   \bottomrule
 \end{tabular}}
\end{center}

\vspace{.3cm}

Assume that $(\varepsilon_{\cc{q.1}}(u^p), \varepsilon_{\cc{q.2}}(u^p)) = (2, -1)$. Then, from the Luthar-Passi-formula (\cref{LuPa}):
\[ \mu_1(u^p, \rho) = \frac{1}{q} \Big( \rho(1) - 2(\rho(1) - q\frac{2n-q-3}{2}) - (\rho(1) - q(2n - 2q - 3)) \Big) = q. \] On the other hand, this has to be equal to $\mu_p(u, \rho) + (p-1)\mu_1(u, \rho)$ forcing $\mu_1(u, \rho) = 0$ and $\mu_p(u, \rho) = q$. Assume now that $(\varepsilon_{\cc{q.1}}(u^p), \varepsilon_{\cc{q.2}}(u^p)) = (1, 0)$ then necessarily $n = p + q - 1$ as remarked above. Then the Luthar-Passi-formula gives
\[ \mu_1(u^p, \rho) = \frac{1}{q} \Big( \rho(1) - (\rho(1) - q\frac{2n-q-3}{2}) \Big) = \frac{2n-q-3}{2} \leq n - 3, \] using that $q \geq 3$. As this has to coincide with $\mu_p(u, \rho) + (p-1)\mu_1(u, \rho)$ and as $p > \frac{n}{2}$ it follows that $\mu_1(u, \rho) \in \{0, 1\}$. Hence we have
 \begin{align*} \mu_1(u, \rho) = 0 &  \\ \text{or } \quad  \mu_1(u, \rho) = 1 & \quad \text{ and } \quad n = p + q -1.\end{align*}
 
 From \cref{prop:mu1pi}\eqref{prop:mu1pqpi} and the Luthar-Passi-formula for the character $\rho$ we get
{\footnotesize 
\begin{align*} \mu_1(u, \pi) & = \frac{1}{pq} \Bigg[ -q \varepsilon_{\cc{q.1}}(u) -2q\varepsilon_{\cc{q.2}}(u) -p \varepsilon_{\cc{p.1}}(u) + q\Big( \varepsilon_{\cc{q.1}}(u^p) + 2\varepsilon_{\cc{q.2}}(u^p) \Big)  +  p   \Bigg],  \\ \mu_1(u, \rho) & = \frac{1}{2pq} \Bigg[  - (2n-q-3) q \varepsilon_{\cc{q.1}}(u)  - 2(2n-2q-3)q \varepsilon_{\cc{q.2}}(u) - (2n-p-3)p \varepsilon_{\cc{p.1}}(u) \\ & \hspace{2cm} + q\Big( (2n-q-3)\varepsilon_{\cc{q.1}}(u^p) + 2 (2n-2q-3)\varepsilon_{\cc{q.2}}(u^p) \Big) + p(2n-p-3) \Bigg]. \end{align*}}
We have the system of linear equations:
\begin{align} \begin{split} \mu_1(u, \pi) & = 0, \\ \mu_1(u, \rho) & = t \in \{0, 1\}, \\ \varepsilon(u) & = 1. \end{split} \label{eq:3_system_linear_equations} \end{align}
Note that the system has in our case a unique solution over the rationals. Solving for $\varepsilon_{\cc{p.1}}(u)$ yields
\[\varepsilon_{\cc{p.1}}(u) = \frac{p(p - q(3-2t))}{(p-q)(p-2q)} \] and hence in the cases $t = 0$ and $t = 1$ this equals \[ \frac{p(p-3q)}{(p-q)(p-2q)} \qquad \text{and} \qquad \frac{p}{p-2q}, \] respectively, where the second possibility can only occur for $n = p + q - 1$. The first value is an integer only if $(p, q) \in \{(3, 2), (5, 3)\}$, cases we excluded at the beginning of the proof. If the second value is an integer, then $p - 2q \in \{\pm 1\}$. If $p -2q = 1$, then %(using that $n = p + q - 1$ in this case)
$n = 3q$, a contradiction. If $p - 2q = -1$, then there is a unique integral solution to the system \eqref{eq:3_system_linear_equations} of linear equations, namely $(\varepsilon_{\cc{q.1}}(u), \varepsilon_{\cc{q.2}}(u), \varepsilon_{\cc{p.1}}(u)) = (1, p, -p)$. If this happens, then we have $p = 2q - 1$ and $n = 3q - 2$. In this case we need to employ another character.  
Set $\tau = \chi_{(n-3,2,1)}$, then we can calculate with the Frobenius' formula \cite[4.10]{FultonHarris} \begin{align*} \tau(1) = & \frac{1}{3}n(n-2)(n-4) \\ \text{and} \quad \tau(\cc{r.j}) = & \frac{(n-rj)}{3}\Big( (n-rj-1)(n-rj-5) + 3 \Big) \end{align*} for $\cc{r.j} \in \{\cc{q.1}, \cc{q.2}, \cc{p.1} \}$.  Then we find \begin{align*} \mu_p(u, \tau) & = \frac{1}{pq}  \Big( \tau(1) - \tau(u^p) + (p-1)\tau(u^q) - (p-1)\tau(u) \Big)  \\ & = \frac{1}{(2q-1)q} \Big( \tau(1) - (2q-1) \tau(\cc{q.1}) - (2q-1)(2q-2) \tau(\cc{q.2}) + 2q(2q-2)\tau(\cc{p.1}) \Big)   \\ & =  -\frac{1}{3} q (q-4), \end{align*} which can only be non-negative if $q \leq 4$, hence $q = 3$ and $p = 7$. But this contradicts the assumptions on $p$ and $q$ in this case. This finishes the proof. \end{proof}

\begin{remark} \begin{itemize} 
\item In the case $\lfloor \frac{n}{q} \rfloor = 2$ there are clearly more possible partial augmentations for elements of order $q$, but only $(\varepsilon_{\cc{q.1}}(u^p), \varepsilon_{\cc{q.2}}(u^p)) \in \{ (2, -1), (1, 0)\}$ are possible when considering $P(u^p)$ for a normalized unit $u$ of order $pq$.
\item For the final contradiction in the case $(n, p, q) = (19, 13, 7)$ one might have used the character $\chi_{(n-3,3)}$ (which is of smaller degree than $\tau$), but that character does not yield a contradiction for bigger $n$.
\end{itemize} \end{remark}

Recall that the Prime Graph Question has an affirmative answer for $A_n$ if $n \leq 10$. As $\ZZ A_n$ is a subring of $\ZZ S_n$ and there is an element of order $pq$ in $A_n$ if and only if there is such an element in $S_n$ for two distinct odd primes $p$ and $q$ we obtain the following corollary.

\begin{corollary}\label{cor_A_n} Let $n \in \ZZ_{\geq 1}$ and $p, q$ two primes with $p, q > \frac{n}{3}$. Then there is an element of order $pq$ in $A_n$ if and only if there is an element of order $pq$ in $\V(\ZZ A_n)$. \end{corollary}

\section{Units of Even Order}

Throughout this section $p$ will denote an odd prime. We will prove certain constraints for units of order $2p$ in $\V(\ZZ S_n)$. In the prime graph of $S_n$ the prime $2$ is connected to any other vertex except in the cases when $n \in \{p, p+1\}$; in this case $2$ and $p$ are not connected. Note that the conjugacy class of involutions $\cc{2.j}$ of $S_n$ is contained in $A_n$ if and only if $j$ is even.

\begin{definition} For a subset $X \subseteq S_n$ and an element $u = \sum_{g \in S_n} u_g g \in \ZZ S_n$ we denote by $\varepsilon_X(u) = \sum_{x \in X} u_x$ the \emph{generalized partial augmentation of $u$ with respect to the subset $X$}. \end{definition}

\begin{lemma}\label{prop:augonAp} Let $n \in \ZZ_{\geq 1}$, $p$ be an odd prime and $u \in \V(\ZZ S_n)$ be of order $2p$. Then $\varepsilon_{A_n}(u) = \varepsilon_{A_n}(u^p) \in \{0, 1\}$. \end{lemma}

\begin{proof} Let $S_n/A_n = \langle t \rangle \simeq C_2$. Consider the natural ring epimorphism $f \colon \ZZ S_n \to \ZZ S_n/A_n \simeq \ZZ C_2$. It maps an element $v \in \ZZ S_n$ to $\varepsilon_{A_n}(v)  + \varepsilon_{S_n \setminus A_n}(v)t$. As the ring $\ZZ C_2$ only has trivial (torsion) units we see that one of the values $\varepsilon_{A_n}(u)$ and $\varepsilon_{S_n \setminus A_n}(u)$ is $1$ and the other one is $0$. Furthermore, for every odd $m \in \ZZ_{\geq 1}$ \begin{align*} \varepsilon_{A_n}(u^m)  + \varepsilon_{S_n \setminus A_n}(u^m)t & = f(u^m) = (f(u))^m = \sum_{j = 0}^m {m \choose j} \varepsilon_{A_n}(u)^{j}\varepsilon_{S_n \setminus A_n}(u)^{m-j}t^{m-j} \\ & = \varepsilon_{A_n}(u)^m  + \varepsilon_{S_n \setminus A_n}(u)^m t. \end{align*} The claim follows.
\end{proof}

For the proof of the next lemma we will use two characters, which are stated for the convenience of the reader in the following table ($j \in \{1, ..., \lfloor \frac{n}{2} \rfloor \}$):

\begin{center}
 \noindent\begin{tabular}{cccc} \toprule
        & $1$ & $\cc{2.j}$ & $\cc{p.1}$ \\ \midrule \\[-.38cm]
   $\pi = \chi_{(n-1,1)} $ & $n -1 $ & $\pi(1) - 2j$  & $\pi(1) - p$ \\
   $\pi\otimes \sgn = \chi_{(2, 1^{n-2})}$ & $n -1 $ & $(-1)^j(\pi(1) - 2j)$  & $\pi(1) - p$ \\
\bottomrule
 \end{tabular}
\end{center}

\begin{lemma}\label{prop:sumzero2} Let $p$ be an odd prime and $u \in \V(\ZZ S_p)$ be of order $2p$. Then $\mu_1(u, \pi) = 0$ and $\sum_{j = 1}^{\lfloor \frac{p}{2} \rfloor} j \varepsilon_{\cc{2.j}} (u^p) = 0$. More precisely, 
\[\sum_{\substack{j = 1 \\ 2 \nmid j}}^{\lfloor \frac{p}{2} \rfloor} j \varepsilon_{\cc{2.j}} (u^p) = 0 \quad \text{ and } \quad \sum_{\substack{j = 1 \\ 2 \mid j}}^{\lfloor \frac{p}{2} \rfloor} j \varepsilon_{\cc{2.j}} (u^p) = 0. \] \end{lemma}

\begin{proof} As there is only one conjugacy class, $\cc{p.1}$, of elements of order $p$ in $S_p$, we know that $u^2$ is rationally conjugate to an element of that class. Hence $P(u^2) \sim \text{diag}\left( \eigbox{1}{\zeta_p, ..., \zeta_p^{p-1} }  \right)$. As $\pi(u) \in \ZZ$ we have \[ P(u) \sim \text{diag}\left( \eigbox{1}{\zeta_p, ..., \zeta_p^{p-1} }  \right) \quad \text{or} \quad P(u) \sim \text{diag}\left( \eigbox{1}{-\zeta_p, ..., -\zeta_p^{p-1} }  \right).\] From the decomposition matrix for the prime $p$ \cite[24.1 Theorem (ii)]{James} we see that $\pi_p$, the restriction of $\pi$ to a $p$-modular Brauer character, has a trivial constituent. As  the order of $u^p$ is not divisible by $p$ (i.e.\  $u^p$ is $p$-regular) we get $\pi(u^p) \geq -(p - 3)$ (see e.g.\ \cite[Proposition 3.1 (i)]{HertweckBrauer}), ruling out the second option for $P(u)$. Thus $\mu_1(u, \pi) = 0$ and $u^p \in \Ker(P)$. The multiplicity of $-1$ as an eigenvalue of $P(u^p)$ was calculated in \cref{prop:mu1pi}\eqref{prop:mu1rpi} to be $\mu_1(u^p, \pi) = \sum_{j = 1}^{\lfloor \frac{n}{2} \rfloor} j \varepsilon_{\cc{2.j}}(u^p)$, using that $P(u^p)$ is the identity, we have $\sum_{j = 1}^{\lfloor \frac{p}{2} \rfloor} j \varepsilon_{\cc{2.j}} (u^p) = 0$. 

Now consider a representation $R$ affording the character $\pi \otimes \sgn = \chi_{(2, 1^{n-2})}$. With the same reasoning as above we get  \[ R(u) \sim \text{diag}\left( \eigbox{1}{\zeta_p, ..., \zeta_p^{p-1} }  \right) \quad \text{or} \quad R(u) \sim \text{diag}\left( \eigbox{1}{-\zeta_p, ..., -\zeta_p^{p-1} }  \right).\] The decomposition matrix for the prime $p$ implies that the restriction of $\pi \otimes \sgn$ to a Brauer character for the prime $p$ contains $\sgn$ as a constituent. Note that $\sgn(u^p)$ only depends on the value $\varepsilon_{A_p}(u) = \varepsilon_{A_p}(u^p) \in \{0, 1\}$; we get that $R(u)$ has the first form given above, if $\varepsilon_{A_p}(u^p) = 1$ and it has the second form given above if $\varepsilon_{A_p}(u^p) = 0$. Hence $\mu_1(u^p, \pi \otimes \sgn)$ equals $0$ if $\varepsilon_{A_p}(u^p) = 1$ and it equals $p-1$ if $\varepsilon_{A_p}(u^p) = 0$. 

For the multiplicity of $-1$ as an eigenvalue of $R(u^p)$ we get from the Luthar-Passi-formula (\cref{LuPa}) and the fact that $\sum_{j = 1}^{\lfloor \frac{p}{2} \rfloor} j \varepsilon_{\cc{2.j}} (u^p) = 0$ \begin{align*} \mu_1(u^p,\pi \otimes \sgn) & = \frac{1}{2} \left[ \pi(1) -  \sum_{j = 1}^{\lfloor \frac{p}{2} \rfloor} (-1)^j(\pi(1) - 2j) \varepsilon_{\cc{2.j}}(u^p) \right] \\ & =   \frac{1}{2} \left[ \sum_{j = 1}^{\lfloor \frac{p}{2} \rfloor} (\pi(1) - 2j) \varepsilon_{\cc{2.j}}(u^p)  -  \sum_{j = 1}^{\lfloor \frac{p}{2} \rfloor} (-1)^j(\pi(1) - 2j) \varepsilon_{\cc{2.j}}(u^p) \right] \\ & =  \sum_{\substack{j = 1 \\ 2 \nmid j}}^{\lfloor \frac{p}{2} \rfloor} (\pi(1) - 2j) \varepsilon_{\cc{2.j}}(u^p) \\ &  =  \pi(1)\varepsilon_{S_p \setminus A_p}(u^p) - 2\sum_{\substack{j = 1 \\ 2 \nmid j}}^{\lfloor \frac{p}{2} \rfloor} j \varepsilon_{\cc{2.j}}(u^p). \end{align*} Plugging in both possibilities for $\varepsilon_{A_p}(u^p)$ and $\mu_1(u^p,\pi \otimes \sgn)$ we get the last two statements of the lemma. \end{proof}

The argument that a certain representation decomposes modulo $p$ was also used in \cite[Example 3.6]{Hert_AlgColloq}.

\section{Applications}

In this section we use \cref{big_p_q,cor_A_n,prop:mu1is0,prop:sumzero2} to verify the Prime Graph Question for certain almost simple groups. Additional calculations were made using the \textsf{HeLP} package \cite{HeLP,HeLP_article} implemented in \textsf{GAP} \cite{GAP}.

\begin{theorem}\label{prop:PQ_for_small_asg} The Prime Graph Question has an affirmative answer for all almost simple groups with socle an alternating group of degree at most $17$. 
\end{theorem}

\begin{proof} The Prime Graph Question has an affirmative answer for alternating groups of degree at most $10$ \cite{SalimA10} and all almost simple groups with socle isomorphic to $A_n$ for $n \leq 6$ \cite{KiKoStAndrews,BaMaM10}. We only have to exclude the existence of a normalized unit of order $pq$ in the integral group ring of the largest symmetric (or alternating) group which does not contain an element of order $pq$. After applying \cref{big_p_q,cor_A_n} we are left with the cases of units of order $pq$ in $\V(\ZZ S_m)$ or $\V(\ZZ A_m)$ for triples $(m, p, q)$ given in \cref{infoSm,infoAm}. We handled these cases with \textsf{GAP} \cite{GAP} and the \textsf{HeLP} package \cite{HeLP,HeLP_article} which uses the software package \textsf{4ti2} \cite{4ti2}.

In this proof only Brauer characters will be used as they are known for all groups involved and they provide us with systems with smaller coefficients, which can be solved more efficiently and which provide stronger constraints. For a prime $r$ the Brauer table of $S_m$ can be obtained in \textsf{GAP} \cite{GAP} by \texttt{CharacterTable("Sm") mod r}. By $\varphi^r_s$ we denote the character number $s$ in that Brauer table.

Let $p$ and $q$ be two primes, $p > q$. Let $G$ be an alternating or a symmetric group of degree $m$ having one conjugacy class of elements of order $p$ and no element of order $pq$. Assume there is a unit $u \in \V(\ZZ G)$ of order $pq$. Let $\psi$ be a Brauer character of $G$ modulo a prime $r$ (different from $p$ and $q$) which is rational valued and let $D$ be a representation affording this character. For Brauer characters there is an analogue of the formula in \cref{LuPa}, see \cite[Section 4]{HertweckBrauer}. Let first $\xi$ be a primitive $q$-th root of unity. Then the formula for the multiplicity $\mu_\ell(u^p, \psi)$ of $\xi^\ell$ as an eigenvalue of $D(u^p)$ is

\begin{equation}\label{eq:mu_up} \mu_\ell(u^p, \psi) = \sum_{j = 1}^{\lfloor \frac{m}{q} \rfloor} a_{\tt{q.j}} \varepsilon_{\tt{q.j}}(u) + b \end{equation}
with $b = \nicefrac{\psi(1)}{q}$ and 
 \[a_{\tt{q.j}} = \begin{cases} \nicefrac{(q-1)\psi(\tt{q.j})}{q} & \text{ if } q \mid \ell \\ 
 - \nicefrac{\psi(\tt{q.j})}{q} & \text{ if } q \nmid \ell.
 \end{cases}
\]

Now assume that $\zeta$ is a primitive $pq$-th root of unity. Then the formula for the multiplicity of $\zeta^\ell$ as an eigenvalue of $D(u)$ is as follows:
\begin{equation}\label{eq:mu_u} \mu_\ell(u, \psi) = \sum_{j = 1}^{\lfloor \frac{m}{q} \rfloor} a_{\tt{q.j}} \varepsilon_{\tt{q.j}}(u) + a_{\tt{p.1}} \varepsilon_{\tt{p.1}}(u) + b \end{equation}
with 
\begin{equation*}
 {\footnotesize
a_{\tt{x.j}} = \begin{cases} \nicefrac{(p-1)(q-1)\psi(\tt{x.j})}{pq} & \text{ if }  pq \mid \ell \\ 
 -\nicefrac{(q-1)\psi(\tt{x.j})}{pq} & \text{ if } q \mid \ell, p\nmid \ell \\
 -\nicefrac{(p-1)\psi(\tt{x.j})}{pq} & \text{ if } p \mid \ell, q\nmid \ell \\
 \nicefrac{\psi(\tt{x.j})}{pq} & \text{ if } (pq, \ell) = 1,
 \end{cases} \qquad
 b = \begin{cases}\frac{1}{pq} (\psi(1) + (p-1)\psi(\texttt{p.1}) + (q-1)\psi(u^p )& \text{ if }  pq \mid \ell \\
 \frac{1}{pq} (\psi(1)- \psi(\texttt{p.1}) + (q-1)\psi(u^p ) & \text{ if } q \mid \ell, p\nmid \ell \\
 \frac{1}{pq} (\psi(1) + (p-1)\psi(\texttt{p.1}) -\psi(u^p ) & \text{ if } p \mid \ell, q\nmid \ell \\
 \frac{1}{pq} (\psi(1)-\psi(\texttt{p.1}) - \psi(u^p ) & \text{ if } (pq, \ell) = 1,
 \end{cases} 
 }
 \end{equation*}
for $\tt{x} \in \{\tt{p}, \tt{q}\}$.

The strategy is now as follows: we will exploit that $\mu_\ell(u^p, \psi)$ has to be a non-negative integer for all characters $\psi$ and all $\ell$ to produce a finite list of possible partial augmentations for elements $u^p$. Then we will apply \cref{prop:mu1is0,prop:sumzero2} (where applicable) to reduce the number of cases which might occur. Afterwards we use that $\mu_\ell(u, \psi)$ has to be a non-negative integer for all characters $\psi$ and all $\ell$ to disprove the existence of the unit $u$.

We give details for one case, the non-existence of units of order $3\cdot 11$ in $\V(\ZZ S_{13})$. The first table contains the coefficients for the formula in \eqref{eq:mu_up} to obtain a finite number of possible partial augmentations for units of order $3$, in this case 141. After applying \cref{prop:mu1is0} we see that only 18 of them might be $11$-th powers of a unit of order $3\cdot 11$. For those remaining possibilities a system of inequalities $\mu_\ell(u, \psi) \in \ZZ_{\geq 0}$ given by the formulas in \eqref{eq:mu_u} is provided which does not admit an integral solution. Several cases are grouped together as the multiplicities of the same eigenvalues under the same representations can be used. As the coefficients $a_{\tt{x.j}}$ do not depend on the element $u^p$, those coefficients are only given once for each group. The constant term $b = b(\psi, \ell)$ does depend on the partial augmentations $\underline{\varepsilon}(u^{11}) = (\varepsilon_{\cc{3.1}} (u^{11}), \varepsilon_{\cc{3.2}} (u^{11}), \varepsilon_{\cc{3.3}} (u^{11}), \varepsilon_{\cc{3.4}} (u^{11}))$, so for every pair $(\psi, \ell)$ which is used, the absolute coefficient $b(\psi, \ell)$ is given in the next table for every partial augmentation of $u^{11}$ in this group.

\vspace{.5cm}

{\footnotesize
$o(u) = 3$ \\
\begin{tabular}{cccccc}\toprule
  $(\psi, \ell)$ & $a_{\tt{3.1}}$ & $a_{\tt{3.2}}$ & $a_{\tt{3.3}}$ & $a_{\tt{3.4}}$  & $b$ \\ \midrule
   $(\varphi_3^2, 0)$ & $-\nicefrac{64}{3}$ & $\nicefrac{32}{3}$ & $-\nicefrac{16}{3}$ & $\nicefrac{8}{3}$ & $\nicefrac{64}{3}$ \\
   $(\varphi_3^2, 1)$ & $\nicefrac{32}{3}$ & $-\nicefrac{16}{3}$ & $\nicefrac{8}{3}$ & $-\nicefrac{4}{3}$ & $\nicefrac{64}{3}$  \\
   $(\varphi_4^2, 0)$ & $\nicefrac{68}{3}$ & $\nicefrac{26}{3}$ & $\nicefrac{2}{3}$ & $-\nicefrac{4}{3}$ & $\nicefrac{64}{3}$  \\
   $(\varphi_4^2, 1)$ & $-\nicefrac{34}{3}$ & $-\nicefrac{13}{3}$ & $-\nicefrac{1}{3}$ & $\nicefrac{2}{3}$ & $\nicefrac{64}{3}$  \\
   $(\varphi_5^2, 0)$ & $-64$ & $16$ & $0$ & $-4$ & $96$  \\
   $(\varphi_6^2, 0)$ & $\nicefrac{152}{3}$ & $\nicefrac{32}{3}$ & $\nicefrac{2}{3}$ & $\nicefrac{8}{3}$ & $\nicefrac{208}{3}$ \\ \bottomrule
\end{tabular} \\
Number of solutions: $141$, number of relevant solutions after applying \cref{prop:mu1is0}: $18$

\vspace{.3cm}

$o(u) = 3 \cdot 11$\\
Group 1 (12 possible partial augmentations for $u^{11}$):\\
\begin{minipage}{.5\textwidth}
\centering
\begin{tabular}{cccccc}\toprule
  $(\psi, \ell)$ & $a_{\tt{3.1}}$ & $a_{\tt{3.2}}$ & $a_{\tt{3.3}}$ & $a_{\tt{3.4}}$ & $a_{\tt{11.1}}$  \\ \midrule
  $(\varphi_3^2, 0)$ & $-\nicefrac{640}{33}$ & $\nicefrac{320}{33}$ & $-\nicefrac{160}{33}$ & $\nicefrac{80}{33}$ &  $-\nicefrac{40}{33}$ \\
  $(\varphi_3^2, 11)$ & $\nicefrac{320}{33}$ &  $-\nicefrac{160}{33}$ &  $\nicefrac{80}{33}$ & $ -\nicefrac{40}{33}$ & $\nicefrac{20}{33}$ \\ \bottomrule
\end{tabular}
\vspace*{.5cm} 

\begin{tabular}{ccc}\toprule
  $\underline{\varepsilon}(u^{11})$ & $b{(\varphi_3^2, 0)}$ & $b{(\varphi_3^2, 11)}$ \\ \midrule
  $(-1, 0, 6, -4)$ &  $-\nicefrac{20}{33}$ & $\nicefrac{76}{33}$ \\
  $( 0, 0, 3, -2)$ & $-\nicefrac{20}{33}$ & $\nicefrac{76}{33}$ \\
  $( 1, 0, 0, 0  )$ & $-\nicefrac{20}{33}$ & $\nicefrac{76}{33}$ \\ \midrule
  \multicolumn{3}{c}{Continued} \\ \midrule
 \end{tabular}
 \end{minipage}
  \hspace{1cm}
 \begin{tabular}{ccc}\midrule
  $\underline{\varepsilon}(u^{11})$ & $b{(\varphi_3^2, 0)}$ & $b{(\varphi_3^2, 11)}$ \\ \midrule 
  $(-1, 1, 5, -4 )$ & $\nicefrac{28}{33}$ & $\nicefrac{52}{33}$ \\
  $( 0, 1, 2, -2)$ &  $\nicefrac{28}{33}$ & $\nicefrac{52}{33}$ \\
  $( 1, 1, -1, 0)$ &  $\nicefrac{28}{33}$ & $\nicefrac{52}{33}$ \\
  $( -1, 2, 3, -3)$ & $\nicefrac{100}{33}$ & $\nicefrac{16}{33}$ \\
  $( 0, 2, 0, -1 )$ &  $\nicefrac{100}{33}$ & $\nicefrac{16}{33}$ \\
  $( 1, 2, -3, 1)$ &  $\nicefrac{100}{33}$ & $\nicefrac{16}{33}$ \\
  $( -1, 2, 2, -2)$ & $\nicefrac{124}{33}$ & $\nicefrac{4}{33}$ \\
  $( 0, 2, -1, 0)$ &  $\nicefrac{124}{33}$ & $\nicefrac{4}{33}$\\
  $( 1, 2, -4, 2)$ &  $\nicefrac{124}{33}$ & $\nicefrac{4}{33}$ \\ \bottomrule
\end{tabular} \\[.3cm]

Group 2 (5 possible partial augmentations for $u^{11}$):\\
\begin{tabular}{cccccc}\toprule
  $(\psi, \ell)$ & $a_{\tt{3.1}}$ & $a_{\tt{3.2}}$ & $a_{\tt{3.3}}$ & $a_{\tt{3.4}}$ & $a_{\tt{11.1}}$  \\ \midrule
  $(\varphi_4^2, 0)$ & $\nicefrac{680}{33}$ & $\nicefrac{260}{33}$ & $\nicefrac{20}{33}$ & $-\nicefrac{40}{33}$ &  $-\nicefrac{40}{33}$ \\  
  $(\varphi_4^2, 11)$ & $-\nicefrac{340}{33}$ &  $-\nicefrac{130}{33}$ &  $-\nicefrac{10}{33}$ & $ \nicefrac{20}{33}$ & $\nicefrac{20}{33}$ \\ \bottomrule 
\end{tabular}
\hspace{1cm}
\begin{tabular}{ccc}\toprule
  $\underline{\varepsilon}(u^{11})$ & $b{(\varphi_4^2, 0)}$ & $b{(\varphi_4^2, 11)}$ \\ \midrule
  $(-1, 1, 4, -3)$ &  $\nicefrac{2}{3}$ & $\nicefrac{5}{3}$ \\
  $(  0, 1, 1, -1)$ & $\nicefrac{76}{33}$ & $\nicefrac{28}{33}$ \\  
  $( 1, 1, -2, 1 )$ & $-\nicefrac{130}{33}$ & $\nicefrac{1}{33}$ \\
  $(-1, 3, 1, -2)$ &  $\nicefrac{64}{33}$ & $\nicefrac{34}{33}$ \\
  $( 0, 3, -2, 0  )$ & $\nicefrac{118}{33}$ & $\nicefrac{7}{33}$ \\ \bottomrule
\end{tabular} \\[.3cm]

Group 3 (1 possible partial augmentation for $u^{11}$):\\
\begin{tabular}{cccccc}\toprule
  $(\psi, \ell)$ & $a_{\tt{3.1}}$ & $a_{\tt{3.2}}$ & $a_{\tt{3.3}}$ & $a_{\tt{3.4}}$ & $a_{\tt{11.1}}$  \\ \midrule
  $(\varphi_2^2, 0)$ & $\nicefrac{60}{11}$ & $\nicefrac{40}{11}$ & $\nicefrac{20}{11}$ & $0$ &  $\nicefrac{20}{33}$ \\  
  $(\varphi_2^2, 11)$ & $-\nicefrac{30}{11}$ &  $-\nicefrac{20}{11}$ &  $-\nicefrac{10}{11}$ & $ 0$ & $-\nicefrac{10}{33}$ \\ 
  $(\varphi_2^5, 0)$ & $-\nicefrac{640}{11}$ & $\nicefrac{160}{11}$ & $0$ & $-\nicefrac{40}{11}$ &  $\nicefrac{40}{33}$ \\ 
  $(\varphi_2^5, 11)$ & $\nicefrac{320}{11}$ &  $-\nicefrac{80}{11}$ &  $0$ & $ \nicefrac{20}{11}$ & $-\nicefrac{20}{33}$  \\ \bottomrule 
\end{tabular}
\hspace{1cm}
\begin{tabular}{ccc}
  \toprule
  $\underline{\varepsilon}(u^{11})$ & $b{(\varphi_2^2, 0)}$ & $b{(\varphi_2^2, 11)}$  \\ \midrule
  $( 1, 3, -5, 2 )$ &  $\nicefrac{46}{33}$ & $\nicefrac{10}{3}$  \\ \cmidrule[1pt]{2-3} & $b{(\varphi_2^5, 0)}$ & $b{(\varphi_2^5, 11)}$ \\ \cmidrule{2-3} & $\nicefrac{236}{33}$ & $\nicefrac{344}{33}$ \\ \bottomrule
\end{tabular}
}

\vspace{.5cm}

For the other cases which are needed to complete the proof of the theorem the information is given in \cref{infoSm,infoAm}. In each row the existence of a unit of order $pq$ in the integral group ring of a symmetric group $S_m$ or an alternating group $A_m$ is disproved. To obtain a finite number of solutions for elements of order $q$, the constraints $\mu_0(u^p, \psi) \in \ZZ_{\geq 0}$ and $\mu_1(u^p, \psi) \in \ZZ_{\geq 0}$ obtained from the characters in column 5 are used; the number of partial augmentations for elements of order $q$ which fulfill these constraints is given in the next column. Where applicable, \cref{prop:mu1is0,prop:sumzero2} are used to reduce the number of those tuples which can be the partial augmentations of the $p$-th power of a normalized unit of order $pq$. The last column indicates which Brauer characters modulo $r$ are sufficient to exclude the existence of units of order $pq$ for all possibilities for partial augmentations which remained. For elements of order $2$ in $\V(\ZZ S_{17})$ the Brauer table mod $17$ and the computer algebra package Normaliz \cite{Normaliz,NormalizArticle} was used. The existence of elements of order $2\cdot 17$ was excluded with the Brauer table modulo $3$. \qedhere
\begin{table}[h] \caption{Data for Units in $\V(\ZZ S_m)$}\label{infoSm}
{\footnotesize
\begin{tabular}{cccccccc}\toprule
  $m$ & $q$ & $p$ & $r$ &\begin{minipage}{2.5cm} Brauer characters mod $r$ used for units of order $q$ \end{minipage} & \begin{minipage}{2.5cm} Number of partial augmentations for units of order $q$ \end{minipage} &  \begin{minipage}{2.5cm} Number after applying \cref{prop:mu1is0,prop:sumzero2} \end{minipage} &  \begin{minipage}{2.5cm} Brauer characters mod $r$ excluding existence of units of order $p\cdot q$ \end{minipage} \\  \midrule
  7 & 3 & 5 &  & \multicolumn{4}{l}{\cref{ex:order_3_5}} \\  % it looks like it is possible to rule out (2, -1) for small symmetric groups with the Brauer table mod 2 but this is no longer possible for elements of order 7 in S_14
  8 & 2 & 7 & 3 & 2, ..., 4 & 16 & 16 & 3, 5 \\
  9 & 3 & 7 & 2 & 2, ..., 4 & 9 & 4 & 4 \\
 12 & 2 & 11 & 3 & 2, ..., 6 & 558 & 558 &  2, ..., 9 \\
 13 & 3 & 11 & 2 & 3, ..., 6 & 128 & 18 &  2, 3, 6 \\
 14 & 2 & 13 & 3 & 2, ..., 10 & 10\ 259 & 10\ 259 & 2, ..., 10 \\
 15 & 3 & 13 & 2 & 2, 3, 16, 17, 19 & 1\ 389 & 363 & 2, 3, 8, 16, 17, 19 \\
 15 & 5 & 11 & 2 & 2, 16 & 19 & 10 &  2, 19\\
 17 & 2 & 17 & 17, 3 & 2, ..., 10 & 570\ 252 & 10\ 004 & 2, ..., 13 \\
 17 & 3 & 17 & 2 & 2, ..., 6 & 3\ 138 & 367 & 2, ..., 7 \\
 17 & 5 & 17 & 2 & 2, 3 & 30 & 6 & 2, ..., 4\\
 17 & 5 & 13 & 2 & 2, 3 & 30 & 12 & 2, ..., 4 \\
\bottomrule
\end{tabular} 
}
\end{table}

\begin{table}[h] \caption{Data for Units in $\V(\ZZ A_m)$}\label{infoAm}
{\footnotesize
\begin{tabular}{ccccccc}\toprule
  $m$ & $q$ & $p$ & $r$ &\begin{minipage}{2.5cm} Brauer characters mod $r$ used for units of order $q$ \end{minipage} & \begin{minipage}{2.5cm} Number of partial augmentations for units of order $q$ \end{minipage} &   \begin{minipage}{2.5cm} Brauer characters mod $r$ excluding existence of units of order $p\cdot q$ \end{minipage} \\  \midrule
 14 & 2 & 11 & 3 & $2, 3$ & 58 &  $2,3, 5$ \\
 16 & 2 & 13 & 3 & $2, 3, 5$ & 1105 &  $2, ..., 7$ \\
\bottomrule
\end{tabular} 
}
\end{table}

\end{proof}

\bibliographystyle{amsalpha}
\bibliography{PQ_Sn.bib}

\end{document}